\documentclass[12pt]{article}

\usepackage{amsmath}
\usepackage{amssymb}
\usepackage{amsthm}
\usepackage{marvosym}

\newtheorem{theorem}{Theorem}
\newtheorem{prop}[theorem]{Proposition}

\newtheorem{cor}[theorem]{Corollary}

\newcommand{\bP}{\mathbf{P}}

\newcommand{\NP}{\mathbf{NP}}
\newcommand{\FP}{\mathbf{FP}}
\newcommand{\GI}{\mathbf{GI}}

\newcommand{\cC}{\mathcal{C}}

\newcommand{\cI}{\mathcal{I}}

\newcommand{\cR}{\mathcal{R}}

\DeclareMathOperator\iw{iw}
\DeclareMathOperator\Gal{Gal}
\allowdisplaybreaks

\title{The Complexity of Counting Poset and Permutation Patterns}
\author{Joshua Cooper and Anna Kirkpatrick \\ Mathematics Department, University of South Carolina \\ 1523 Greene St., Columbia SC 29208}

\begin{document}

\maketitle


\begin{abstract}
We introduce a notion of pattern occurrence that generalizes both classical permutation patterns as well as poset containment.  Many questions about pattern statistics and avoidance generalize naturally to this setting, and we focus on functional complexity problems -- particularly those that arise by constraining the order dimensions of the pattern and text posets.  We show that counting the number of induced, injective occurrences among dimension 2 posets is \(\#\bP\)-hard; enumerating the linear extensions that occur in realizers of dimension 2 posets can be done in polynomial time, while for unconstrained dimension it is \(\GI\)-complete; counting not necessarily induced, injective occurrences among dimension 2 posets is \(\#\bP\)-hard; counting injective or not necessarily injective occurrences of an arbitrary pattern in a dimension 1 text is \(\#\bP\)-hard, although it is in \(\FP\) if the pattern poset is constrained to have bounded intrinsic width; and counting injective occurrences of a dimension 1 pattern in an arbitrary text is \(\#\bP\)-hard, while it is in \(\FP\) for bounded dimension texts.  This framework easily leads to a number of open questions, chief among which are (1) is it \(\#\bP\)-hard to count the number of occurrences of a dimension 2 pattern in a dimension 1 text, and (2) is it \(\#\bP\)-hard to count the number of texts which avoid a given pattern?
\end{abstract}

\section{Introduction}

A tremendous amount of study has been dedicated to understanding occurrence or non-occurrence of combinatorial substructures: which substructures are avoidable and counting objects that avoid them, what substructures random and random-like objects possess, enumerating substructures in general objects, describing the subclass of objects that have particular substructure counts, etc.  Much interesting work (particularly in model theory) has concerned the completely general question of substructure occurrence, but the degree of abstraction involved changes the nature of which questions are useful to ask.  However, here we investigate a somewhat more specific perspective that still allows us to address important questions from two disparate but highly studied areas: permutation patterns and subposet containment.  In order to illustrate this unifying viewpoint, we first describe these two topics:

\begin{enumerate}

\item  A ``permutation of \(n\)'' is a bijection from \([n] = \{1,\ldots,n\}\) to itself for some positive integer \(n\).  Given a permutation \(\sigma\), called a ``pattern'', and another permutation \(\tau\) of \(n\), called the ``text'', we say that \(\sigma\) ``occurs on'' or ``matches'' the index set \(\cI \subset [n]\) in \(\tau\) if \(\tau|_\cI\) is order-isomorphic to \(\sigma\); that is, if \(\cI = \{a_1,\ldots,a_k\}\) with \(a_1 < \cdots < a_k\), for any \(i, j \in [k]\), \(\sigma(i) < \sigma(j)\) if and only if \(\tau(a_i) < \tau(a_j)\).  Interesting questions about pattern occurrence include the complexity of counting the number of occurrences of a pattern, the distribution of pattern counts for random permutations, enumeration of permutations which avoid a given pattern, and the structure of permutations with specified pattern counts.  

\item A ``poset of size \(n\)'' is a set (the ``ground set'') of cardinality \(n\) associated with a partial order of the set, that is, a binary relation which is reflexive, antisymmetric, and transitive.  Given posets \(P = (S, \prec_P)\) and \(Q = (T, \prec_Q)\), we say that \(Q\) contains \(P\) (as a ``subposet'') on the set \(U \subset T\) if there is an order-preserving bijection between \(P\) and \(Q|_U\); that is, if there exists a bijection \(f : S \rightarrow U\) so that, for \(x, y \in S\), \(x \prec_P y\) implies \(f(x) \prec_Q f(y)\).  Furthermore, the containment is said to be ``induced'' if the implication is in fact biconditional, and ``unlabelled'' if it is understood only up to order automorphisms of \(P\).  Interesting questions about poset containment include the complexity of counting the number of subposets of a given type, the size of the largest subposet of a given poset not containing a fixed subposet, and the nature of linear extensions of a poset (which correspond in the present language to occurrences of a poset in a chain -- a total order -- of equal size).

\end{enumerate}

To unify these perspectives, we introduce a notion of pattern occurrence that generalizes classical permutation pattern matching as well as poset containment.  In the most general formulation, let \(P\) and \(Q\) be posets which we term the ``pattern'' and ``text'' posets, respectively.  We say that \(P\) ``occurs at'' a subposet \(Q^\prime\) of \(Q\) if there exists an onto function \(f : P \rightarrow Q^\prime\) so that \(f\) is order-preserving, i.e., \(v \prec_P w\) implies \(f(v) \prec_Q f(w)\); in this case, \(f\) is called an ``occurrence'' of \(P\) in \(Q\).  Furthermore, we say that the occurrence is ``induced'' if \(f\) is in fact an order isomorphism, i.e., \(f\) is an occurrence so that \(f(v) \prec_Q f(w)\) implies \(v \prec_P w\), and we say that the occurrence is ``injective'' (``bijective'') if \(f\) is injective (respectively, bijective).  One can also speak of ``unlabeled'' pattern occurrences as equivalence classes of occurrences from a pattern \(P\) to a text \(Q\) modulo automorphisms of \(P\).

We reformat several classical problems in the language of permutation patterns using the notion of order dimension (sometimes called ``Dushnik-Miller dimension'' \cite{DM41}).  Given a poset \(P\), a ``linear extension'' of \(P\) is a bijective occurrence of \(P\) in a chain \(C\) (a totally ordered poset).  We use an equivalent definition from poset theory interchangeably with this: a linear extension is a total order \(\prec^\prime\) on the ground set of \(P\) so that \(v \prec_P w\) implies \(v \prec^\prime w\) for any \(v,w \in P\).  A family \(\cR = \{f_1,\ldots,f_r\}\) of linear extensions of \(P\) is said to be a ``realizer'' of \(P\) if a relation \(v \prec w\) is in \(P\) iff \(f_j(v) \prec_C f_j(w)\) for every \(j \in [r]\); a realizer \(\cR\) is ``minimal'' iff it has the fewest possible number of elements among all realizers of \(P\); the cardinality of a minimal realizer of \(P\) is the ``dimension'' \(\dim(P)\) of \(P\).  The ``width'' of a poset is the size of its largest antichain, i.e., subset of vertices between which there are no relations.  An ``automorphism'' of a poset is a bijective occurrence of a poset in itself.

We also refer to standard texts in computational complexity theory to precisely define hardness of decision and functional complexity problems (e.g., \cite{AB09}).  Roughly, a decision problem is in \(\bP\) if the answer can be obtained in polynomial time (in the size of the input instance); it is in \(\NP\) if the answer can be certified in polynomial time; it is \(\NP\)-hard if every problem in \(\NP\) can be reduced to it in polynomial time (i.e., it is at least as hard as all problems in \(\NP\)); it is \(\NP\)-complete if it is \(\NP\)-hard and in \(\NP\).  Similarly, a function problem (a computational problem whose output is an integer instead of only a single bit) is in \(\FP\) if the answer can be obtained in polynomial time (in the size of the input instance); it is in \(\#\bP\) if it consists of computing the number of correct solutions to a problem in \(\NP\); it is \(\#\bP\)-hard if the problem of computing the number of correct solutions to any problem in \(\NP\) can be reduced to this problem in polynomial time; it is \(\#\bP\)-complete if it is \(\#\bP\)-hard and in \(\#\bP\).

\section{Results}

In the following, we denote by \(P\) a pattern poset and by \(Q\) a text poset.

\begin{theorem} \label{thm:2to2}
If \(\dim(P) = \dim(Q) = 2\), the problem of computing the number of unlabeled, induced, injective occurrences of \(P\) in \(Q\) is \(\#\bP\)-hard.
\end{theorem}

Given a permutation \(\sigma\) of \(n\), there is a poset \(D(\sigma)\) associated with \(\sigma\), whose ground set is \([n]\) and \(i \prec_{D(\sigma)} j\) if and only if \(i < j\) and \(\sigma(i) < \sigma(j)\).  In other words, \(D(\sigma)\) is the two-dimensional poset with realizer comprised of the ordinary total order \(\leq\) on \([n]\) and the pullback \(\sigma^\ast(\leq)\).  We can define an automorphism of \(\sigma\) simply to be an automorphism of \(D(\sigma)\). A not-necessarily-induced match of a permutation pattern \(\sigma\) in a permutation text \(\tau\) is an occurrence of \(D(\sigma)\) in \(D(\tau)\); in the language of permutations, these are maps between the corresponding index sets that preserve coinversions but not necessarily inversions.  (For a permutation \(\sigma \in \mathfrak{S}_n\) and \(i, j \in [n]\), the pair \(\{i,j\}\) is said to be an ``inversion'' if \((i-j)(\sigma(i)-\sigma(j)) < 0\) and a ``coinversion'' if \((i-j)(\sigma(i)-\sigma(j)) > 0\).)


\begin{prop} \label{prop:countingauts} The problem of counting the number of automorphisms of a dimension two poset is in \(\FP\); equivalently, the problem of counting the number of automorphisms of a permutation is in \(\FP\).
\end{prop}

Note the constrast with poset automorphism counting in general.  Indeed, poset automorphism counting is at least as hard as bipartite poset automorphism counting, which is easy to see is polynomial-time equivalent to bipartite graph automorphism counting; bipartite graph isomorphism counting is known to be as hard as general graph isomorphism counting by, for example, \cite{KTZ82}.  By \cite{M79}, this is polynomial-time reducible to the graph isomorphism decision problem, and is therefore so-called ``\(\GI\)-complete'', a complexity class widely believed to be strictly between \(\bP\) and \(\NP\)-hard.

\begin{theorem} \label{thm:lab_ind_inj_2-to-2}
If \(\dim(P) = \dim(Q) = 2\), the problem of computing the number of labeled, induced, injective occurrences of \(P\) in \(Q\) is \(\#\bP\)-hard.
\end{theorem}

\begin{cor}\label{cor:losethedimension}
For any pattern \(P\) and text \(Q\), the problem of computing the number of (labeled or unlabeled) induced, injective occurrences of \(P\) in \(Q\) is \(\#\bP\)-hard.
\end{cor}


\begin{theorem} \label{thm:modifiedBBL}
If \(\dim(P) = \dim(Q) = 2\), the problem of computing the number of (labeled or unlabeled) not necessarily induced, injective occurrences of \(P\) in \(Q\) is \(\#\bP\)-hard.
\end{theorem}

\begin{cor} \label{cor:NPhard} Deciding whether a given dimension \(2\) poset has a not necessarily induced, injective, unlabeled match in another dimension \(2\) poset is \(\NP\)-complete.
\end{cor}

\begin{theorem} \label{thm:BWisjustdimQis1}
If \(\dim(Q) = 1\), the problem of counting the number of (injective or not necessarily injective) occurrences of an arbitrary \(P\) in \(Q\) is \(\#\bP\)-hard.
\end{theorem}

This essentially a restatement of Brightwell and Winkler's famous result that counting the number of linear extensions of a poset is \(\#\bP\)-hard.  By contrast, some special cases of this problem are in fact easy.  Before proceeding, we define the ``(Gallai) modular decomposition''\footnote{Unfortunately, there are quite a few names in the literature given to modules in addition to ``modules'': ``autonomous sets'', ``intervals'', ``homogeneous sets'', ``partitive sets'', and ``clans'', for example.} of a poset.  Given \(P = (S,\prec)\), define a subset \(T \subset S\) to be a ``module'' of \(P\) if, for all \(u,v \in T\) and \(x \in S \setminus T\), \(u \prec x\) iff \(v \prec x\) and \(x \prec u\) iff \(x \prec v\).  A module \(T\) is ``strong'' if, for any module \(U \subset S\), \(U \cap T \neq \emptyset\)
implies \(U \subset T\) or \(T \subset U\).  Thus, the nonempty strong modules of \(P\) form a tree order, called the ``(Gallai) modular decomposition'' of \(P\).  A strong module or poset is said to be ``indecomposable'' if its only proper submodules are singletons and the empty set.  It is a result of Gallai (\cite{G67}) that the maximal proper strong modules of \(P\) are a partition \(\Gal(P)\) of \(T\), and it is straightforward to see that the quotient poset \(P / \Gal(P)\) is well-defined.  Furthermore, Gallai showed the following.  The comparability graph \(G(P)\) of a poset \(P\) has as its vertex set the ground set of \(P\) and has an edge \(\{x,y\}\) for \(x \neq y\) if \(x \prec_P y\) or \(y \prec_P x\).

\begin{theorem}[Gallai \cite{G67}] Given a poset \(P\) such that \(|P| \geq 2\), one of the following holds.
\begin{enumerate}
\item (Parallel-Type) If \(G(P)\) is not connected, then \(\Gal(P)\) is the family of subposets induced by the connected components of \(G(P)\) and \(P/\Gal(P)\) is an empty poset.
\item (Series-Type) If the complement \(\overline{G(P)}\) of \(G(P)\) is not connected, then \(\Gal(P)\) is the family of subposets induced by the connected components of \(\overline{G(P)}\) and \(P/Gal(P)\) is a linear order.
\item (Indecomposable) Otherwise, \(|Gal(P)| \geq 4\) and \(P/Gal(P)\) is indecomposable.
\end{enumerate}
\end{theorem}

Define the ``intrinsic width'' \(\iw(P)\) of a poset as the maximum width of its indecomposable modules.  (So, for example, series-parallel posets are characterized by having intrinsic width \(1\).)  The following strengthens a result of Steiner (\cite{S90}), who provides a similar, albeit incomplete, proof of a slightly weaker result.

\begin{theorem} \label{thm:betterthanSteiner}
If the intrinsic width of a poset is bounded, its number of linear extensions (i.e., bijective occurrences as a pattern in a dimension \(1\) text poset) can be computed in polynomial time.  In particular, given a chain \(Q\), if \(\iw(P) \leq k\), there is an algorithm that computes in \(O(n^{\max(3,k)})\) time the number of occurrences of \(P\) in \(Q\).
\end{theorem}

\begin{theorem} \label{thm:Pisdim1} If \(\dim(P) = 1\), then counting the number of injective occurrences of \(P\) in an arbitrary poset \(Q\) is \(\#\bP\)-hard.
\end{theorem}

\section{Proofs}

One of the main reasons that order dimension gives rise to interesting questions about poset pattern occurrence is the fact that unlabeled, induced, injective pattern occurrence corresponds in a precise way to permutation pattern matching when both pattern and text have dimension \(2\).  In particular, suppose \(\dim(P) = \dim(Q) = 2\), and let \(P\) have a realizer consisting of \(([k],<)\) and \(([k],\prec_P^\ast)\) and \(Q\) has a realizer consisting of \(([n],<)\) and \(([n],\prec_Q^\ast)\).  One can think of \(P\) as representing the permutation \(\sigma_P\) of \(k\) such that \(\sigma_P(i) < \sigma_P(j)\) iff \(i \prec_P^\ast j\) and of \(Q\) as similarly representing a \(\sigma_Q\) of \(n\).  Then, we have the following rubric connecting permutation pattern matching to poset patterns occurrence in dimension 2.

\begin{prop} \label{prop:2to2isjustpermmatching} The matches of \(\sigma_P\) in \(\sigma_Q\) are in bijection with the unlabeled, induced, injective occurrences of \(P\) in \(Q\).
\end{prop}
\begin{proof} Note that the matches of \(\sigma_P\) in \(\sigma_Q\) correspond exactly with certain subsets of \([n]\) of size \(k\), namely, those \(\cI \in \binom{[n]}{k}\) so that \(\sigma_Q|_{\cI}\) is order-isomorphic to \(\sigma_P\).  Suppose \(\cI = \{r_1 < \cdots < r_k\}\), and define \(f : [k] \rightarrow [n]\) by \(f(i) = r_i\).  We claim that \(f\) provides an isomorphism between \(P\) and \(Q|_{f(\cI)}\), and thus is an induced, injective occurrence of \(P\) in \(Q\) on the set \(f(\cI)\).  Note that, for \(i, j \in [k]\), if \(i \prec_P j\), then \(i < j\) and \(\sigma_P(i) < \sigma_P(j)\), whence \(r_i < r_j\) and \(\sigma_Q(r_i) < \sigma_Q(r_j)\), so \(r_i = f(i) \prec_Q f(j) = r_j\); furthermore, this argument is reversible, so \(i \prec_P j\) iff \(f(i) \prec_Q f(j)\).

We now show that this map from matches of \(\sigma_P\) in \(\sigma_Q\) to occurrences of \(P\) in \(Q\) is unique up to automorphisms of \(P\).  Suppose that \(g\) and \(h\) are induced, injective occurrences of \(P\) in \(Q\) with \(g([k]) = h([k])\), so \(i \prec_P j\) iff \(g(i) \prec_Q g(j)\) iff \(h(i) \prec_Q h(j)\).  Define \(\tau = h^{-1} \circ g\).  We claim that \(\tau\) is an automorphism of \(P\).  Indeed, suppose \(i,j \in [k]\); we wish to show that \(\tau(i) \prec_P \tau(j)\) iff \(i \prec_P j\).  Indeed, \(\tau(i) \prec_P \tau(j)\) iff \(h^{-1}(g(i)) \prec_P h^{-1}(g(j))\) iff \(h(h^{-1}(g(i))) \prec_Q h(h^{-1}(g(j)))\) iff \(g(i) \prec_Q g(j)\) iff \(i \prec_P j\).  Therefore, matches from \(\sigma_P\) to \(\sigma_Q\) correspond bijectively to equivalence classes under automorphisms of \(P\) of induced, injective occurrences of \(P\) in \(Q\), i.e., unlabeled, induced, injective occurrences of \(P\) in \(Q\).
\end{proof}

It is easy to see that the not necessarily induced matches of \(\sigma_P\) in \(\sigma_Q\) are also in bijection with the not necessarily induced, unlabeled, injective occurrences of \(P\) in \(Q\).

\begin{proof}[Proof of Theorem \ref{thm:2to2}.] Note that, by Proposition \ref{prop:2to2isjustpermmatching}, the problem of computing the number of unlabeled, induced, injective occurrences of \(P\) in \(Q\) is polynomial-time reducible to the problem of counting matches of \(\sigma_P\) in \(\sigma_Q\).  Since the pattern and text here are arbitrary, by \cite{BBL98}, this is a \(\#\bP\)-hard computational problem.
\end{proof}

The next proof closely resembles in some aspects the argument for Theorem 5 of \cite{IR06}.

\begin{proof}[Proof of Proposition \ref{prop:countingauts}.] Let \(M_1,\ldots,M_k\) be the indecomposable strong modules of \(P\).  By Theorem 4.2 of \cite{BEK94}, each \(P[M_i]\) has at most two automorphisms.  In particular, if \((\prec_i^1,\prec_i^2)\) is a realizer of \(P[M_i]\) (unique up to ordering by \cite{G67}), and \(\tau\) is the permutation so that \(a \prec_i^1 b\) iff \(\tau_i(a) \prec_i^2 \tau_i(b)\), then \(P[M_i]\) has either no nontrivial automorphisms, or else \(\tau_i\) is its only one.  Note that it is certainly polynomial-time to check if \(\tau_i\) is indeed an automorphism; let \(t \leq k \leq |P|\) be the number of such \(i\), so that computing \(2^t\) is in \(\FP\).
It is now straightforward to describe all automorphisms of \(P\).  Since automorphisms preserve (strong) modules, all automorphisms of \(P\) arise as automorphisms of the indecomposable strong modules composed with automorphisms of the tree corresponding to the Gallai decomposition.  Furthermore, series-type nodes have only trivial automorphisms, while parallel-type nodes can be arbitrarily reordered, so the number of automorphisms has size
\[
2^t \prod_{P_0 \subset P} |\Gal(P_0)|!
\] 
where the \(P_0\) vary over all parallel-type strong modules of the Gallai decomposition.  By \cite{BM83}, it is possible to compute the entire Gallai decomposition in polynomial time; since
\[
\log \prod_{P_0 \subset P} |\Gal(P_0)|! < \sum_{P_0 \subset P} |\Gal(P_0)|^2 \leq (\sum_{P_0 \subset P} |\Gal(P_0)|)^2 \leq 4 |P|^2,
\]
this shows that computing the number of automorphisms of a dimension 2 poset is in \(\FP\).  (The total number of vertices of a rooted tree none of which have exactly one child is at most twice the number of leafs.)
\end{proof}

\begin{proof}[Proof of Theorem \ref{thm:lab_ind_inj_2-to-2}.]  By Proposition \ref{prop:countingauts} the problem of computing the number of labeled, induced, injective occurrences of one dimension \(2\) poset in another is polynomial-time reducible to the problem of computing the number of unlabeled, induced, injective occurrences of one dimension \(2\) poset in another.  This latter problem is \(\#\bP\)-hard by Theorem \ref{thm:2to2}.
\end{proof}

\begin{proof}[Proof of Corollary \ref{cor:losethedimension}]  This follows immediately from Theorems \ref{thm:2to2} and \ref{thm:lab_ind_inj_2-to-2}, since the problem without dimension constraints is more general.
\end{proof}


The next proof involves a modification of the argument of \cite{BBL98}, and in fact can be used to provide another proof of the \(\#\bP\)-hardness of permutation pattern matching because all of the matches involved are in fact induced.

\begin{proof}[Proof of Theorem \ref{thm:modifiedBBL}]  Since not necessarily induced, unlabeled, injective occurrences of a dimension \(2\) poset \(P\) in a dimension \(2\) poset \(Q\) are equivalent by Proposition \ref{prop:2to2isjustpermmatching} to not necessarily induced matches of \(\sigma_P\) in \(\sigma_Q\), we show that the latter problem is \(\#\bP\)-hard.  Suppose \(\Sigma\) is an instance of 3-SAT over \(n\) variables \(\{x_1,\ldots,x_n\}\), i.e.,
\[
\Sigma = C_1 \wedge \cdots \wedge C_m
\]
where \(C_i = (v^i_{1} \vee v^i_{2} \vee v^i_{3})\), each \(v^i_{j}\) being a literal of the form \(x_{a(i,j)}\) or \(\neg x_{a(i,j)}\), \(a(i,j) \in [n]\).  We assume that no variable occurs both positively and negatively in the same clause.  Define a pattern \(\pi\) and text \(\tau\) permutation as follows.  More correctly, for convenience of notation, we define two sequences of distinct reals which can be interpreted as permutations.  We treat sequences and words interchangeably, writing concatenation as \((\cdot)\)-product.  Then
\[
\pi = \pi^x_{1} \cdots \pi^x_{n} \cdot \pi^C_{1} \cdots \pi^C_{m}
\]
and
\[
\tau = \tau^x_{1} \cdots \tau^x_{n} \cdot \tau^C_{1} \cdots \tau^C_{m}.
\]
Define
\[
\pi^x_i = (2n+2i-1) \cdot i \cdot (2n-i+1) \cdot (2n+2i),
\]
and
\begin{align*}
\tau^x_i &= (4n+4i-1) \cdot (2i-1) \cdot (4n-2i+2) \cdot (4n+4i) \\
& \qquad \cdot (4n+4i-3) \cdot (2i) \cdot (4n-2i+1) \cdot (4n+4i-2).
\end{align*}
We need a few more definitions before describing the \(\pi_i^C\) and \(\tau_i^C\), \(1 \leq i \leq m\).  In particular, we describe \(\pi_i^C\) inductively, that is, once \(\pi_1^C\) through \(\pi_{i-1}^C\) have been described.  Let \(u_{ij}\) for each \(j \in [3]\) be any real number strictly between \(a(i,j)\) and \(2n - a(i,j) + 1\) so that \(u_{ij}\) is strictly larger than \(u_{i^\prime j}\) for each \(1 \leq i^\prime < i\).  Define
\[
\pi^C_i = (4n + 2i - 1) \cdot u_{i1} \cdot u_{i2} \cdot u_{i3} \cdot (4n + 2i).
\]
Now, we describe \(\tau_i^C\) inductively.  Let \(T_{j}\) for each \(j \in [n]\) be the open interval \((2j-1,4n-2j+2)\); let \(F_j\) be the open interval \((2j,4n-2j+1)\).  Now, for each \(j \in [n]\), let \(T_{ij} = T_j\) if \(x_j\) occurs positively in \(C_i\) and \(T_{ij} = F_j\) if \(x_j\) occurs negatively in \(C_i\); similarly, let \(F_{ij} = F_j\) or \(F_{ij} = T_j\) if \(x_j\) occurs positively or negatively in \(C_i\), respectively.  (We do not define \(T_{ij}\) or \(F_{ij}\) if \(x_j\) does not occur in \(C_i\).) For each \(x_j\) that appears in \(C_i\), choose \(t_{ijk} \in T_{ij}\) for each \(k \in [4]\) so that
\[
t_{ij1} < t_{ij2} < t_{ij3} < t_{ij4}
\]
and \(t_{ij1} > t_{i^\prime j 4}\) for each \(1 \leq i^\prime < i\).  Next, for each \(x_j\) that appears in \(C_i\), choose \(f_{ijk} \in T_{ij}\) for each \(k \in [3]\) so that
\[
f_{ij1} < f_{ij2} < f_{ij3}
\]
and \(f_{ij1} > f_{i^\prime j 3}\) for each \(1 \leq i^\prime < i\).  Finally,
\begin{align*}
\tau^C_i = & (8n + 14i - 1) \cdot q_0(i) \cdot (8n + 14i) \\
& (8n + 14i - 3) \cdot q_1(i) \cdot (8n + 14i - 2) \\
& (8n + 14i - 5) \cdot q_2(i) \cdot (8n + 14i - 4) \\
& (8n + 14i - 7) \cdot q_3(i) \cdot (8n + 14i - 6) \\
& (8n + 14i - 9) \cdot q_4(i) \cdot (8n + 14i - 8) \\
& (8n + 14i - 11) \cdot q_5(i) \cdot (8n + 14i - 10) \\
& (8n + 14i - 13) \cdot q_6(i) \cdot (8n + 14i - 12).
\end{align*}
where
\begin{align*}
q_i(0) = & t_{ia(i,1)1} \cdot t_{ia(i,2)1} \cdot t_{ia(i,3)1} \\
q_i(1) = & t_{ia(i,1)2} \cdot t_{ia(i,2)2} \cdot f_{ia(i,3)1} \\
q_i(2) = & t_{ia(i,1)3} \cdot f_{ia(i,2)1} \cdot t_{ia(i,3)2} \\
q_i(3) = & t_{ia(i,1)4} \cdot f_{ia(i,2)2} \cdot f_{ia(i,3)2} \\
q_i(4) = & f_{ia(i,1)1} \cdot t_{ia(i,2)3} \cdot t_{ia(i,3)3} \\
q_i(5) = & f_{ia(i,1)2} \cdot t_{ia(i,2)4} \cdot f_{ia(i,3)3} \\
q_i(6) = & f_{ia(i,1)3} \cdot f_{ia(i,2)3} \cdot t_{ia(i,3)4}.
\end{align*}
We claim that satisfying assignments of \(\Sigma\) are in bijection with matches of \(\pi\) in \(\tau\).

\begin{description}
\item[Claim 1.] Consider any not necessarily induced match of \(\pi\) into \(\tau\).  We claim that \(\pi^x_i\) matches into \(\tau^x_i\) for each \(i \in [n]\) and \(\pi^C_j\) matches into \(\tau^C_j\) for each \(j \in [m]\).  Note that the following is an increasing subsequence of \(\pi\) of length \(2n+2m\):
\[
\pi_0 = (2n+1) \cdot (2n+2) \cdots (4n+2m-1) \cdot (4n+2m).
\]
Let \(z_k\), for each \(k \in [2n+2m]\), be the index so that \(\pi(z_k) = \pi_0(k)\).  Suppose that \(\{\tau(z^\prime_k)\}_{k=1}^{2n+2m}\), with \(z^\prime_{k+1} > z^\prime_{k}\) for each \(k \in [2n+2m-1]\), is an increasing subsequence \(\tau_0\) of \(\tau\) that can occur as the image of \(\pi_0\) in some match of \(\pi\) into \(\tau\).  Then, for each \(k \in [2n+2m-1]\), we must have
\[
z^\prime_{k+1} - z^\prime_k \geq z_{k+1} - z_{k}.
\]
It is straightforward to see that every such \(\tau_0\) has the form
\[
\tau_0(r_1,\ldots,r_n;s_1,\ldots,s_m) = w^x_1 \cdots w^x_n \cdot w^C_1 \cdots w^C_m
\]
where \(w^x_i\) is a subsequence of \(\tau^x_i\) of the form \((4n+4i-1-2r_i)\cdot(4n+4i-2r_i)\) for some \(r_i \in \{0,1\}\) and \(w^C_i\) is a subsequence of \(\tau^C_i\) of the form \((8n+14i-1-2s_i)\cdot(8n-14i-2s_i)\) for some \(s_i \in \{0,1,2,3,4,5,6\}\).
\item[Claim 2.] Consider any not necessarily induced match of \(\pi\) into \(\tau\); we claim it has a very particular structure, described as follows.  By Claim 1, \(\pi_0\) matches precisely some \(\tau_0\).  First, \(i \cdot (2n-i+1)\) must match to \((2i-1+r_i) \cdot (4n+2i+2-r_i)\).  Then, \(u_{i1} \cdot u_{i2} \cdot u_{i3}\) must match to \(q_i(s_i)\).  These positions are forced because there are precisely two (respectively, three) elements of the sequence between the elements \(2n+2i-1\) and \(2n+2i\) for each \(i \in [n]\) in \(\pi_0\) and between the elements of \(w^x_i\) in \(\tau_0\) (respectively, \(4n+2i-1\) and \(4n+2i\) for each \(i \in [m]\) in \(\pi_0\) and between the elements of \(w^C_i\) in \(\tau_0\)).  Furthermore, it is straightforward to see that any such map from \(\pi\) to \(\tau\) is indeed an (induced!) match, and in fact, the \(s_i\)'s are determined by the \(r_i\)'s. 
\item[Claim 3.] We claim that matches of \(\pi\) into \(\tau\), as described above, are in bijection with satisfying assignments.  Given a match of \(\pi\) into \(\tau\), the corresponding assignment sets \(x_i\) equal to true if \(r_i = 0\) and false if \(r_i = 1\); the clause \(C_i\) is satisfied by the assignment because \(v_{ij}\) receives the value \(\top\) if the \(j\)-th binary digit of \(s_i\) is \(0\) and \(\bot\) otherwise, and \(s_i \in \{0,1,2,3,4,5,6\}\).  Finally, it is clear that every satisfying assignment arises from such a match by choosing the \(r_i\)'s to reflect the appropriate variable settings.
\end{description}

\end{proof}

Because the matches used in the previous proof may be considered not necessarily induced, and they correspond exactly to satisfying assignments of the 3-CNF formula involved, Corollary \ref{cor:NPhard} follows immediately.

\begin{proof}[Proof of Theorem \ref{thm:BWisjustdimQis1}] Note that the number of injective occurrences of \(P\) in \(Q\) is just the number of linear extensions of \(P\) times \(\binom{|Q|}{|P|}\), the latter quantity being computable in polynomial time, and the former being \(\#\bP\)-hard by the result \cite{BW91} of Brightwell-Winkler. The number of not necessarily injective occurrences of \(P\) in \(Q\) is the number of linear extensions of \(P\) times \(\binom{|Q|+|P|-1}{|P|}\), the latter quantity being computable in polynomial time, and the former being \(\#\bP\)-hard by the result \cite{BW91} of Brightwell-Winkler.
\end{proof}

Before proceeding, note that a ``down-set'' of a poset \(P\) is simply a set \(D\) of elements of \(P\) so that \(x \in D\), \(y \in P\), and \(y \prec_P x\) implies \(y \in D\); we write \(x \lessdot_P y\) if \(y\) covers \(x\) in \(P\), i.e., \(x \prec_P y\) and there is no \(z \in P\) so that \(x \prec_P z\) and \(z \prec_P y\).

\begin{proof}[Proof of Theorem \ref{thm:betterthanSteiner}.]  First, suppose \(P\) has width \(k\) (a constant) and cardinality \(n\).  Then, by Dilworth's Theorem, there is a chain decomposition \(\cC = \{C_1,\ldots,C_k\}\); as mentioned in \cite{K90}, there are well-known \(O(n^3)\) algorithms for computing the Dilworth decomposition.  Now, we construct the lattice \(L\) of down-sets of \(P\) from \(\cC\), keeping track of \(|D \cap C_i|\) for each \(i\) and the children \(D^\prime \lessdot_L D\) of \(D\) as we construct the down-sets \(D\).  Starting from the empty set (which is the minimal element of \(L\)), we iteratively build up all down-sets by considering the least unused element of each \(C_i\) one at a time.  That is, given some down-set \(D\), we test if \(D \cup \{x_i\}\) is also a down-set for each \(x_i\), the least element of \(C_i\) which does not appear in \(D\), by checking if \(x_i\) satisfies \(y_j \nprec x_i\) for each \(j \in [k]\), where \(y_j\) is the element (if it exists) of height \(|D \cap C_j|+1\) in \(C_j\); this requires at most \(k^2\) comparisons per down-set \(D\).  It is straightforward to update the \(|D \cap C_j|\) and child lists appropriately for each new down-set.  Since down-sets are uniquely determined by the quantities \(|D \cap C_i|\), \(i \in [k]\), there are at most \(n^k\) such \(D\).  Therefore, the algorithm so far has cost \(O(n^{\max(3,k)})\) time.

Next, we use \(L\) to compute the number of linear extensions of \(P\).  Let \(f(D)\), for a down-set \(D\) of \(P\), denote the number of linear extensions of \(D\); it is easy to see that
\[
f(D) = \sum_{D^\prime \lessdot_L D} f(D^\prime),
\]
a calculation that requires summing at most a constant (\(k\)) number of integers at each step.  Again, the number of down-sets is at most \(n^k\), so we obtain \(f(P)\), the desired quantity, in time \(O(n^{\max(3,k)})\).

By \cite{MS94}, it is possible to compute the Gallai decomposition of a poset in \(O(n^2)\) time.  If the indecomposable modules of \(P\) have cardinalities \(m_1,\ldots,m_t\), the time to compute the number of linear extensions of all of them is, by the above argument, at most
\[
O \left (\sum_{i=1}^t m_i^{\max(3,k)} \right ) \leq O \left ( \left (\sum_{i=1}^t m_i \right )^{\max(3,k)} \right ) = O(n^{\max(3,k)}).
\]
Once the number of linear extensions of all the indecomposable modules has been computed, we combine these numbers into the number of linear extensions of \(P\) by recursing on the nodes of the Gallai decomposition.  Indeed, if a node is ``series-type'', then the number of linear extensions of the corresponding module is simply the product of the number of linear extensions of its children; if a node is ``parallel-type'' with children of cardinalities \(m^\prime_1,\ldots,m^\prime_u\), then the number of linear extensions of the corresponding module is the product of the number of linear extensions of its children and the quantity
\[
\binom{m^\prime_1 + \cdots + m^\prime_u}{m^\prime_1,\ldots,m^\prime_u}.
\] 
It is easy to check that the numerical computations involved require at most \(O(n^{\max(3,k)})\) time, so one can compute the total number of linear extensions of \(P\) in this amount of time.
\end{proof}

\begin{proof}[Proof of Theorem \ref{thm:Pisdim1}] This also follows from the main result of \cite{BW91}.  If we let \(R\) be any poset, \(P\) a chain of length \(|R|\), and \(Q\) the lattice of down-sets of \(R\), then the number of injective occurrences of \(P\) in \(Q\) is precisely the number of linear extensions of \(R\), which is \(\#\bP\)-hard to compute.
\end{proof}

Note that counting the number of injective occurrences of a chain in an arbitrary text poset of bounded dimension is in \(\FP\), because the standard dynamic programming algorithm for counting increasing subsequences of permutation readily generalizes to arbitrary dimension and executes in polynomial time.

\section{Conclusion and Problems}

A plethora of questions remain open about poset pattern occurrence.  First, if \(\dim(P)=2\) and \(Q\) is a chain, is the problem of counting the number of (not necessarily induced, injective) occurrences of \(P\) in \(Q\) \(\#\bP\)-hard?  In the language of permutations, this is the computational problem of counting, for a given permutation \(\sigma\), how many permutations of the same length have all of the inversions of \(\sigma\) (and possibly others as well).  The problem is also equivalent to asking whether counting linear extensions of dimension \(2\) posets is hard, a question left open by \cite{BW91} because the posets whose number of linear extensions the authors compute grow in dimension without bound.  Indeed, their gadget \(Q_I(p)\) contains as an induced subposet a bipartite poset whose ``upper set'' is the family of clauses that occur in the 3-SAT instance \(I\) and whose ``lower set'' is the family of variables occurring in \(I\), with a edge between a clause and variable precisely when the clause contains a literal which is the positive or negative of the variable.  If we choose the clauses to be all possible disjunctions of two variables out of \(n\) (repeating the second to ensure three literals), the resulting subposet is exactly the subposet of the Boolean poset between the doubletons and singletons; Spencer showed (\cite{S71}), via a result of Dushnik, that this ``Boolean layer'' poset has dimension \(\Omega(\log \log n)\), and therefore in particular tends to infinity.

What are the complexity of poset pattern recognition problems for other parameter settings than those considered here?  What problems/results from the substantial literature on permutation pattern avoidance generalize in an interesting way to not necessarily induced occurrences?  Is it \(\#\bP\)-hard to compute the number of posets of cardinality \(n\) which avoid (i.e., do not contain any occurrence of) given patterns, especially, in the case when the dimension of pattern and text is 2?  Given the apparent difficulty of computing the number of pattern-avoiding permutations for various special patterns (1324 being a prominent example), it is natural to suspect that this problem is computationally hard in general.


\begin{thebibliography}{100}
\bibitem{AB09} S.~Arora, B.~Barak, {\it Computational complexity. A modern approach.} Cambridge University Press, Cambridge, 2009.
\bibitem{BM83} H.~Buer, R.~H.~M\"{o}hring, A fast algorithm for the decomposition of graphs and posets.
{\it Math.~Oper.~Res.}~{\bf 8} (1983), no.~2, 170--184. 
\bibitem{BBL98} P.~Bose, J.~Buss, A.~Lubiw, Pattern matching for permutations, {\it Inf.~Proc.~Let.}~{\bf 65} (1998), 277--283.
\bibitem{BEK94} B.~I.~Bayoumi, M.~H.~El-Zahar, S.~M.~Khamis, Counting two-dimensional posets. {\it Discrete Math.}~{\bf 131} (1994), no.~1--3, 29--37.
\bibitem{BW91} G.~Brightwell, P.~Winkler, Counting linear extensions. {\it Order} {\bf 8} (1991), no.~3, 225--242.
\bibitem{DM41} B.~Dushnik, E.~W.~Miller, Partially ordered sets. {\it Amer.~J.~Math.} {\bf 63} (1941), 600--610.
\bibitem{IR06} P.~Ille, J.-X.~Rampon, A counting of the minimal realizations of the posets of dimension two. {\it Ars Combin.}~{\bf 78} (2006), 157--165.
\bibitem{G67} T.~Gallai, Transitiv orientierbare Graphen. {\it Acta Math.~Acad.~Sci.~Hungar.}~{\bf 18} (1967), 25--66.
\bibitem{K90} H.~Kierstead, Effective versions of the chain decomposition theorem. The Dilworth theorems, 36--38, {\it Contemp.~Mathematicians}, Birkh\"{a}user Boston, Boston, MA, 1990. 
\bibitem{KTZ82} N.~M.~Korneenko, R.~I.~Tyshkevich, V.~N.~Zemlyachenko, The graph isomorphism problem. The theory of the complexity of computations, I. {\it Zap.~Nauchn.~Sem.~Leningrad.~Otdel.~Mat.~Inst.~Steklov.~(LOMI)} {\bf 118} (1982), 83--158, 215. 
\bibitem{M79} R.~Mathon, A note on the graph isomorphism counting problem. {\it Inform.~Process.~Lett.}~{\bf 8} (1979), no.~3, 131--132. 
\bibitem{MS94} R.~M.~McConnell, J.~P.~Spinrad, Linear-time modular decomposition and efficient transitive orientation of comparability graphs. {\it Proceedings of the Fifth Annual ACM-SIAM Symposium on Discrete Algorithms} (1994), 536--545. 
\bibitem{S71} J.~Spencer, Minimal scrambling sets of simple orders. {\it Acta Math.~Acad.~Sci.~Hungar.}~{\bf 22} (1971/72), 349--353. 
\bibitem{S90} G.~Steiner, Polynomial algorithms to count linear extensions in certain posets. {\it Congr.~Numer.}~{\bf 75} (1990), 71--90.
\end{thebibliography}
\end{document}